\documentclass[a4paper, reqno, 11pt]{amsart}

\usepackage[english]{babel}
\usepackage{amsmath, amssymb, amsthm, amsfonts, mathtools} 
\usepackage{graphicx, caption, xcolor, placeins} 
\usepackage{enumerate}
\usepackage{ifthen}
\usepackage{bbm}
\usepackage{geometry}
\provideboolean{shownotes} 
\setboolean{shownotes}{true}

\usepackage{color}

\usepackage{tikz} 
\usetikzlibrary{arrows.meta} 
\tikzset{font=\small, 
point/.style={fill, circle, inner sep=1.2pt}, 
>={Straight Barb[round,angle=60:1.2mm 1]} 
} 

\usepackage{hyperref}
\usepackage{setspace}
\usepackage{soul}

\usepackage{xpatch}
\makeatletter
\patchcmd{\@tocline}{\hfil}
{\nobreak\leaders\hbox{\ifnum#1<2\hfill\else$\m@th%
\mkern 4.5 mu\hbox{.}\mkern 4.5 mu$\fi}\hfill\nobreak}{}{}
\def\l@section{\@tocline{1}{10pt}{1pc}{}{\bfseries}}
\def\l@subsection{\@tocline{2}{0pt}{\dimexpr 1pc+2em}{}{}}
\makeatother

\newcommand{\hole}[1]{
\ifthenelse{\boolean{shownotes}}%
{\begin{center} \fbox{ \rule {.25cm}{0cm}
\rule[-.1cm]{0cm}{.4cm} \parbox{.85\textwidth}{\begin{center}
\texttt{#1}\end{center}} \rule {.25cm}{0cm}}\end{center}}
{}
}

\newtheorem{theorem}{Theorem}[section]
\newtheorem{proposition}[theorem]{Proposition}
\newtheorem{lemma}[theorem]{Lemma}

\allowdisplaybreaks
\theoremstyle{remark}
\newtheorem{remark}[theorem]{Remark}

\newcommand{\R}{\mathbb{R}}
\newcommand{\T}{\mathbb{T}^d}

\newcommand{\cR}{\mathcal{R}}
\newcommand{\cS}{\mathcal{S}}

\newcommand{\dive}{\mathop{\mathrm {div}}}
\newcommand{\curl}{\mathop{\mathrm {curl}}}

\newcommand{\del}{\partial}

\newcommand{\Tone}{\mathbb{T}^1}

\numberwithin{equation}{section}

\begin{document}

\title[Oscillations in Hyperbolic-Parabolic Systems]{Sustained Oscillations in Hyperbolic-Parabolic Systems}

\author[A.E. Tzavaras]{Athanasios E. Tzavaras}
\address[Athanasios E. Tzavaras]{
\newline
Computer, Electrical and Mathematical Science and Engineering Division 
\newline
King Abdullah University of Science and Technology (KAUST)
\newline 
Thuwal 23955-6900,  Saudi Arabia
}
\email{athanasios.tzavaras@kaust.edu.sa}

\baselineskip=18pt

\begin{abstract}
We construct examples of oscillating solutions with persistent oscillations for various  hyperbolic-parabolic systems
with singular diffusion matrices that appear in mechanics. These include, an example for the equations of nonlinear viscoelasticity
of Kelvin-Voigt type with stored energy that violates rank-one convexity, which amounts to a time-dependent variant of twinning solutions. 
An example pertaining to the system of gas dynamics with thermal effects for a viscous, adiabatic gas. 
Finally, an example for the compressible Navier-Stokes system in one-space dimension with nonmonotone pressure function.
We also study the existence of oscillating solutions for linear hyperbolic-parabolic systems with singular diffusion matrices.

\end{abstract}

\maketitle

\tableofcontents

\section{Introduction}

The effectiveness of dissipation on smoothing of discontinuities was proposed by Dafermos \cite{Dafermos81} as 
a yardstick for classifying the hyperbolic-parabolic systems of thermomechanics.
The properties of the Green function for the linearized operator around a constant state capture the efficacy of dissipation, as exemplified 
in the penetrating study of Liu and Zeng \cite{LZ97}.
The picture is less understood regarding propagation of oscillations.  It is well-known that systems of hyperbolic conservation laws can propagate 
oscillations from the initial-data to solutions, and  propagation and/or cancellations of oscillations and its relation to entropy inequalities
has been an extensively studied subject. By contrast, fully-parabolic systems instantaneously smoothen initial oscillations.
Less is known for the intermediate case of hyperbolic-parabolic systems which is the objective of the present work. 

This question was brought to focus by an example, exhibiting sustained oscillations induced by the initial data, \cite[sec 7]{KLST23}, 
for the hyperbolic-parabolic system 
\begin{equation}
\label{eq:onedvisco}
\begin{aligned}
u_t &= v_x
\\
v_t &= \sigma(u)_x + v_{xx}\, ,
\end{aligned}
\end{equation}
with  $\sigma(u)$ non-monotone. The example was produced  to
highlight the gap among two available existence theories for motions $y : (0,T) \times \T \to \R^d$
on the torus $\T$, $d=2,3$, describing the Cauchy problem of nonlinear viscoelastic materials
\begin{align}\label{eq:wave}
&\partial_{tt}y-\dive  \Big (  \frac{\del W}{\del F} (\nabla y) \Big ) = \Delta \partial_{t}y \, , 
\\
\nonumber
&y|_{t=0}=y_0,\quad  \partial_{t}y|_{t=0}=v_0 \, .
\end{align}
Introducing the velocity $v = \del_t y$ and the deformation gradient $F = \nabla y$,  the second order evolution \eqref{eq:wave} 
is expressed as a hyperbolic-parabolic system
\begin{equation}\label{eq:main}
\begin{aligned}
\partial_{t}v-\dive(S(F)) &= \Delta\,v \\
\partial_{t}F-\nabla\,v&=0\\
\curl\,F&=0\, .
\end{aligned}
\end{equation}	
There are  two frameworks  for global existence of weak solutions for  \eqref{eq:main}:
(a) one with spatial regularity $v \in L^2$, $F \in L^p$ in \cite{FD97,Demoulini00};   and (b) a second with spatial regularity 
$v \in L^2$, $F \in H^1$ in \cite{KLST23}. A key and perhaps unexpected difference among the two
is the possibility of sustained oscillations in the framework $v \in L^2$, $F \in L^p$ demonstrated by the example in \cite[sec 7]{KLST23}.

A second motivation stems from the existence theory of the compressible Navier-Stokes system
\begin{equation}
\label{eq:compNS}
\begin{aligned}
\del_t \rho + \dive \rho u &= 0
\\
\del_t \rho u + \dive \rho u \otimes u  + \nabla p (\rho) &=  \dive \Big ( \mu (\nabla u + \nabla u^T) + \lambda (\dive u )  \, I \, \Big )
\end{aligned}
\end{equation}
where $\rho$ and $u$ are the density and velocity of the fluid while $\mu$ the shear viscosity and $\lambda$ the second viscosity 
(or Lame constants) satisfy $\mu > 0$, $\lambda + \mu > 0$.
The existence theory of weak solutions of Lions \cite{b-Lions98} and Feireisl \cite{b-Feireisl04} makes essential use of the propagation 
of compactness property from the initial data. Lions conjectured that the compressible Navier-Stokes system may exhibit propagation of
oscillations and provided an example for a variant of \eqref{eq:compNS} that includes a forcing term, see \cite[Rmk 5.8]{b-Lions98}. 
We will provide an example of propagation of oscillations for the (unforced) compressible Navier-Stokes system in one space dimension. 
The question of deriving the effective equations of the associated homogenization problem is studied in \cite{Serre91, Hillairet07}.

Our objective is to explore persistent oscillations in hyperbolic-parabolic systems. The oscillations 
are related to the fact that systems of mechanics are second order evolutions for the motion
and, as the reader will observe, this plays a fundamental role in the examples that we provide. We explore the existence
of sustained oscillations first for linear systems and then for nonlinear systems.
We start in section \ref{sec:linosc} with linear hyperbolic-parabolic systems,
\begin{equation}
\del_t U + A \del_x U = B \del_{xx} U \, ,
\end{equation}
in one or several space dimensions, where $B$ is singular. We focus on systems originating in mechanics and explore a special structure of oscillations
natural to such problems, where oscillations of the deformation gradient balance with oscillations of the strain-rate.
Examples include the equations of  linear thermoviscoelasticity in one dimension \eqref{lintve1}  or several space dimensions \eqref{lintved}.

In section  \ref{sec:sust}, we consider the system \eqref{eq:main} of nonlinear viscoelasticity  with a stored energy violating rank-1 convexity; 
such models are employed in the theory of phase transitions.
We provide a novel class of dynamic oscillating solutions where the deformation gradient oscillates between two time dependent homogeneous 
states and exhibits jumps across steady interfaces. Such solutions propagate oscillations of initial data to solutions and lead to Young measures
that are convex combinations of Dirac masses. In section \ref{sec:1dnonmon} we study the system  of one-dimensional viscoelasticity
\eqref{longonev}  (or one-dimensional gas dynamics for viscous gases) where the calculations can be made explicit. Similar behavior
appears in a class of one-dimensional models in viscoplasticity studied in section \ref{sec:mdvisco}.

The existence of oscillatory solutions is  extended in section \ref{sec:1dthermoel} to the system describing gas dynamics with thermal effects 
for viscous but adiabatic gases,
\begin{equation}
\label{intro:vhcg}
\begin{aligned}
u_t - v_x &= 0
\\
v_t - \sigma(u,\theta)_x &= ( \frac{\mu}{u} v_x )_x 
\\
\big ( \tfrac{1}{2} v^2 + e(u, \theta) \big )_t - \Big ( \sigma(u,\theta) \, v \Big )_x 
&= \Big ( \frac{\mu}{u} \,  v \, v_x \Big )_x  \, .
\end{aligned}
\end{equation}
For a special constitutive class \eqref{assform} and a carefully selected nonmonotone pressure weak solutions are constructed 
that exhibit persistent oscillations. The construction combines a class of universal solutions describing uniform extension (or compression)  solutions alternating 
via jumps across steady interfaces that satisfy the Rankine-Hugoniot jump condition.

Finally, in section \ref{sec:comprns}, we consider the one-dimensional compressible Navier-Stokes system
\begin{equation}
\label{intro:compNS}
\begin{aligned}
\rho_t  + (\rho u)_y  &= 0
\\
(\rho u_t  + ( \rho u^2  +  p (\rho) )_y  &=  \mu u_{y y} \, ,
\end{aligned}
\end{equation}
By exploring the transformation between Lagrangian and Eulerian descriptions, we map the Lagrangian variable solutions obtained in section \ref{sec:1dnonmon}
to the Eulerian domain. This leads to a special solution of the compressible Navier-Stokes system \eqref{intro:compNS},
\begin{align}
u(t,y) = \frac{y}{t} \, , \qquad 
\rho(t,y) &= 
\begin{cases} \;  \frac{1}{ta}  \quad & \quad  k v_0 (\theta) < \frac{y}{t}   <  k v_0 (\theta) + a \theta   \\[5pt]
                      \;  \frac{1}{tb}  \quad &  \; \; k v_0 (\theta) + a \theta  < \frac{y}{t}   < (k + 1) v_0 (\theta) 
\end{cases}
\quad k \in \mathbb{Z} \, ,
\label{intro:eqnrho}
\end{align}
where $v_0 (\theta) =  \theta a + (1-\theta) b$, $\theta \in (0,1)$,  with persistent oscillations. 
The function $(\rho, u)$ solves \eqref{intro:compNS} provided the nonmonotone pressure satisfies
$p(\frac{1}{ta}) = p(\frac{1}{tb})$ for $t \in [1.2]$. The solution suggests  propagation of oscillations from the initial data  at $t=1$ to the domain $[1,2] \times \R$.

\section{Oscillatory Solutions in Linear Hyperbolic-Parabolic Systems}\label{sec:linosc}

In this section we study linear systems that combine hyperbolic and parabolic characters first in one-space
dimension and then in several space dimensions. The goal is to build a class of solutions for which initial
oscillations propagate into solutions despite the smoothing effect of parabolic characters, exploiting
degeneracies of the parabolic characters. The form of solutions is motivated by examples that arise
in the theory of thermoviscoelasticity and relate to the mechanical interpretation of such models.

In the first two subsections we study the  one-dimensional case, and then we extend the results to the 
linear systems in several space dimensions.

\subsection{Oscillations in a 1-d linear thermoviscoelasticity} \label{sec:osc1d}
Consider the system
\begin{equation}
\label{lintve1}
\begin{aligned}
\frac{\del^2 u}{\del t^2} &= \lambda \frac{\del^2 u}{\del x^2} + m \frac{\del \theta}{\del x} + \mu \frac{\del^2}{\del x^2} \frac{\del u}{\del t}
\\[3pt]
\frac{\del \theta}{\del t} &= \kappa \frac{\del^2 \theta}{\del x^2} + m \frac{\del^2 u}{\del x \del t}
\end{aligned}
\end{equation}
where $(u(t,x), \theta(t,x))$ describe a thermomechanical process consisting of the motion $u$ and the temperature $\theta$,
while $\lambda > 0$, $m \in \R$, $\mu > 0$, $\kappa > 0$ are parameters describing  the elasticity modulus, thermoelastic coupling, 
viscosity and heat diffusivity, respectively. The system \eqref{lintve1} is a second order evolution
but may be expressed as a first order (in time) system by introducing
the velocity  $v = \frac{\del u}{\del t}$ and the elastic strain $w = \frac{\del u}{\del x}$ and writing it as
\begin{equation}
\label{lintve2}
\begin{aligned}
\frac{\del w}{\del t}   &= \; \frac{\del v}{\del x}
\\
\frac{\del v}{\del t} &= \lambda \frac{\del w}{\del x} + m \frac{\del \theta}{\del x} + \mu \frac{\del^2 v }{\del x^2} 
\\
\frac{\del \theta}{\del t} &=  m \frac{\del v}{\del x } + \kappa \frac{\del^2 \theta}{\del x^2}
\end{aligned}
\end{equation}
The form \eqref{lintve2} indicates a hyperbolic-parabolic system, but we will be using mostly the form 
\eqref{lintve1} of a second-order evolution. Solutions of \eqref{lintve1} satisfy the energy identity
\begin{equation}
\begin{aligned}
\frac{\del}{\del t} \left ( \tfrac{1}{2} \Big | \frac{\del u}{\del t} \Big |^2 +   \tfrac{\lambda}{2} \Big | \frac{\del u}{\del x} \Big |^2  + \tfrac{1}{2} | \theta  |^2   \right )
&= 
\frac{\del}{\del x} \left (  \frac{\del u}{\del t} \big (  \lambda  \frac{\del u}{\del x}  +  m  \theta + \mu  \frac{\del^2 u}{\del x \del t}  \big )  \right )
\\
&\quad + \frac{\del}{\del x} \left ( \kappa \theta \frac{\del \theta}{\del x} \right ) - \kappa \Big | \frac{\del \theta }{\del x} \Big |^2  - \mu \Big | \frac{\del^2 u}{\del x \del t}  \Big |^2 
\end{aligned}
\end{equation}
Under appropriate boundary conditions (either periodic or no mechanical work and adiabatic at the boundaries) we obtain the energy
dissipation identity
\begin{equation}
\label{energy1}
\frac{d}{dt} \int \Big ( \tfrac{1}{2} \Big | \frac{\del u}{\del t} \Big |^2 +   \tfrac{\lambda}{2} \Big | \frac{\del u}{\del x} \Big |^2  + \tfrac{1}{2} | \theta  |^2  \Big )\, dx
+ \int \kappa \Big | \frac{\del \theta }{\del x} \Big |^2  + \mu \Big | \frac{\del^2 u}{\del x \del t}  \Big |^2  = 0 \, .
\end{equation}

Next, consider the system \eqref{lintve1} and apply the {\it ansatz} for solutions of the form
\begin{equation}
\label{osc1}
\begin{aligned}
u(t,x) &= \frac{1}{n} a(t) e^{ i n x}
\\
\theta(t,x)  &=   - i b  (t) e^{ i n x}
\end{aligned}
\end{equation}
where $i$ is the imaginary unit, $n$ an integer, and $(a(t), - i b(t))$ are complex amplitude functions.
Introducing \eqref{osc1} to \eqref{lintve1} we deduce that 
 $(a(t), b(t))$ satisfies the system of ordinary differential equations
\begin{equation}
\label{ode1}
\begin{aligned}
\ddot{a} &= - \lambda n^2  a + m n^2 b - \mu n^2   \dot a
\\
\dot{b} &= - \kappa n^2 b  - m  \dot{a}
\end{aligned}
\end{equation}
The amplitude $(a, b)$ in \eqref{ode1} was introduced as complex-valued but it is also consistent to select it real-valued, and
we opt for that selection here.

In summary, given $(a, b )$ a real-valued solution of \eqref{ode1} then $(u(t,x), \theta(t,x))$ selected via the ansatz 
\eqref{osc1} produces an oscillatory solution of \eqref{lintve1}. Taking the real (or imaginary) part of \eqref{osc1},
we obtain a real-valued solution of  \eqref{lintve1}, which for large $n$,  might be highly oscillatory, depending on the 
behavior of $(a, b)$ as a function of $n$. The energy of solutions to \eqref{ode1}
gets damped according to the identity
\begin{equation}
\frac{d}{dt} \left ( \tfrac{1}{2} \dot{a}^2 + \tfrac{\lambda n^2}{2} a^2 + \tfrac{n^2}{2} b^2 \right ) + \mu n^2 \dot{a}^2 + \kappa n^4 b^2 = 0 \, .
\end{equation}
The latter is an equation for the Fourier modes \eqref{osc1} inherited from \eqref{energy1}. 
Observe that for the ansatz \eqref{osc1} we have
$$
\begin{aligned}
\frac{\del u}{\del x} &= i a(t) e^{i n x}
\\
\frac{\del u}{\del t} &= \frac{1}{n} \dot{a}(t) e^{i n x}
\\
\frac{\del^2 u}{\del x \del t} &= i \dot{a}(t) e^{i n x}
\end{aligned}
$$
We expect  oscillations of $u(t,x)$ and  the velocity $v = \frac{\del u}{\del t}(t,x)$ decay to zero (for $n$ large)
while oscillations of the strain $w = \frac{\del u}{\del x}(t,x)$, the strain rate $\frac{\del v}{\del x} = \frac{\del^2 u}{\del x \del t}(t,x)$ and the temperature $\theta(t,x)$
induced by the initial data might persist in time. This has to be validated by studying the behavior of solutions of \eqref{ode1}.

Setting $v = \dot{a}$, \eqref{ode1} is expressed as a first order system,
\begin{equation}
\label{ode2}
\frac{d}{dt} 
\begin{bmatrix}
a \\ v \\ b
\end{bmatrix}
=
\begin{bmatrix}
0 & 1 & 0 \\
- \lambda n^2 &  - \mu n^2   &   m n^2  \\
0 & - m & - \kappa n^2 
\end{bmatrix}
\begin{bmatrix}
a \\ v \\ b
\end{bmatrix}
\end{equation}
Let $A = A(n)$ be the coefficient matrix. The eigenvalues $\rho (n)$ of $A(n)$ are computed as 
roots of the cubic polynomial $\det (A - \rho I) = 0$ which reads
\begin{equation}
\label{algeq1}
\rho^3 + (\kappa + \mu) n^2 \rho^2 + (\kappa \mu n^4 + \lambda n^2 + m^2 n^2 ) \rho + \kappa \lambda n^4 = 0 \, .
\end{equation}
The three roots $\rho_1, \rho_2 , \rho_3$  satisfy the Vieta relations
\begin{equation}
\begin{aligned}
\rho_1 +\rho_2 + \rho_3 &= - (\kappa + \mu) n^2
\\
\rho_1 \rho_2 + \rho_2 \rho_3 + \rho_3 \rho_1  &=  \kappa \mu n^4 + (\lambda + m^2) n^2
\\
\rho_1 \rho_2 \rho_3 &= - \kappa \lambda n^4
\end{aligned}
\end{equation}

While we cannot identify the explicit form of the eigenvalues $\rho_1, \rho_2, \rho_3$, it is possible
to do an asymptotic analysis for large $n$. Setting $r := \frac{\rho}{n^2}$, then $r$ are the roots of the  polynomial
\begin{equation}
\label{algeq2}
r^3 + (\kappa + \mu) r^2 + \Big ( \kappa \mu  +  (\lambda  + m^2) \frac{1}{n^2} \Big ) \rho + \frac{\kappa \lambda}{n^2} = 0
\end{equation}
Introduce the asymptotic expansion $r =  r^0 + \frac{1}{n^2} r^1 + \frac{1}{n^4} r^2 + ...$
to the algebraic equation \eqref{algeq2}; after equating the terms of the same orders we obtain
\begin{align}
\label{asymeq1}
&\mbox{at $O(1)$} \quad &&(r^0)^3 + (\kappa + \mu ) (r^0)^2 + \kappa \mu r^0 = 0
\\
\label{asymeq2}
&\mbox{at $O(\tfrac{1}{n^2})$} \quad && \big (  3 ( r^0)^2 + 2 (\kappa + \mu ) r^0 + \kappa \mu \big )  r^1 = - (\lambda + m^2) r^0  - \kappa \lambda
\\
\label{asymeq3}
&\mbox{at $O(\tfrac{1}{n^4})$} \quad && \big ( 3 ( r^0)^2 + 2 (\kappa + \mu ) r^0 + \kappa \mu \big )  r^2 = - 3 r^0 (r^1)^2 - (\kappa + \mu ) (r^1)^2  - (\lambda + m^2) r^1 
\end{align}
From \eqref{asymeq1} the leading terms in the expansion of the roots are $r^0_1 = 0$, $r^0_2 = - \kappa$, $r^0_3 = - \mu$. Since these are distinct,
for $n$ large the roots are real and we proceed to compute their asymptotic expansions
by solving consecutively \eqref{asymeq2}, \eqref{asymeq3} in the terms $r^1$, $r^2$ etc. 
For instance, for $r^0_1 = 0$ we obtain $ r_1^1 = -\frac{\lambda}{\mu}$ and in turn
$r_1^2 = - \frac{\lambda}{\kappa \mu^2} \big (\frac{\kappa \lambda}{\mu} - m^2 \big )$. 
After a straightforward calculation, using $\rho = r n^2$, 
they produce the following asymptotic expansions for the roots $\rho_i (n)$, $n=1, 2, 3$:
\begin{align}
\rho_1 (n) &= \qquad \qquad -\frac{\lambda}{\mu} + \Big ( - \frac{\lambda}{\kappa \mu^2} \big (\frac{\kappa \lambda}{\mu} - m^2 \big ) \Big ) \frac{1}{n^2}   
+ O \Big (\frac{1}{n^4} \Big)
\label{asymev1}
\\
\rho_2(n) &= - \kappa n^2 - \frac{m^2}{\mu - \kappa} +  \Big ( - \frac{m^2}{\kappa (\mu-\kappa)^2} \big ( \lambda + \kappa \frac{m^2}{\mu-\kappa} \big ) \Big ) \frac{1}{n^2}
+ O \Big (\frac{1}{n^4} \Big)
\label{asymev2}
\\
\rho_3(n) &= - \mu n^2+  \Big ( \frac{\lambda}{\mu} +  \frac{m^2}{\mu - \kappa} \Big ) + 
 \Big ( \big ( \frac{\lambda}{\mu}  +  \frac{m^2}{\mu - \kappa} \big )  \big ( \frac{\lambda}{\mu^2} + \kappa \frac{m^2}{\mu-\kappa} \big ) \Big )   \frac{1}{n^2}
+ O \Big (\frac{1}{n^4} \Big)
\label{asymev3}
\end{align}
The reader can easily verify that the first two terms in the expansion of the roots satisfy the Vieta relations 
within  $O \Big ( \frac{1}{n^2} \Big )$.

Next, we focus on the eigenvalue $\rho_1(n)$ in \eqref{asymev1} and the associated eigenvector $\xi_1 (n)$ of the matrix $A(n)$ in \eqref{ode2}. 
Introduce the expansions 
$$
\begin{aligned}
\rho_1 (n) &= r_1^1 + \frac{1}{n^2} r_1^2 + O \Big (\frac{1}{n^4} \Big )  \quad \mbox{ where $ r_1^1 = -\frac{\lambda}{\mu}$, $r_1^2 = - \frac{\lambda}{\kappa \mu^2} \big (\frac{\kappa \lambda}{\mu} - m^2 \big )$ }
\\
\xi_1 (n) &= \xi_1^0 + \frac{1}{n^2}  \xi_1^1 + O \Big (\frac{1}{n^4}  \Big )
\end{aligned}
$$
into the equation $A(n) \xi_1 (n) = \rho_1 (n) \xi_1 (n)$ to obtain

$$
\Bigg \{  n^2 
\underbrace{ 
\begin{bmatrix}
0 & 0 & 0 \\
- \lambda  &  - \mu    &   m   \\
0 & 0 & - \kappa 
\end{bmatrix}
}_{A_0}
+
\underbrace{
\begin{bmatrix}
0 & 1 & 0 \\
0 &  0  &   0 \\
0 & - m & 0
\end{bmatrix}
}_{A_1}
- \big ( r_1^1 + \frac{1}{n^2} r_1^2 + O \big (\frac{1}{n^4} \big ) \big ) I 
\Bigg \}  \left (\xi_1^0 + \frac{1}{n^2}  \xi_1^1 + O \Big (\frac{1}{n^4}  \Big ) \right ) = 0
$$
Equating terms of the same order,
$$
\begin{aligned}
A_0 \xi_1^0 &= 0   &&\mbox{at order $O(n^2)$}
\\
A_0 \xi_1^1 &= - (A_1 - r_1^1 I ) \xi_1^0 &&\mbox{at order $O(1)$}
\\
A_0 \xi_1^2 &= - (A_1 - r_1^1 I ) \xi_1^1 + r_1^2 \xi_1^0 \quad &&\mbox{at order $O(n^{-2})$} \, ,
\end{aligned}
$$
we compute
$$
\xi_1^0 = \begin{bmatrix} \mu  \\ -\lambda \\ 0  \end{bmatrix}
\, , \quad
\xi_1^1 = \begin{bmatrix} - \frac{\lambda}{\mu}  \\[5pt]  \frac{m^2 \lambda}{\kappa \mu}  \\[5pt] \frac{m \lambda}{\kappa}  \end{bmatrix}
$$

The associated solution of the system \eqref{ode2},
$$
\begin{bmatrix}
a \\ v \\ b
\end{bmatrix} (t)
=
\left ( \xi_1^0 + \frac{1}{n^2} \xi_1^1 + O\big ( \frac{1}{n^4} \big ) \right )  \exp  \left \{ -\frac{\lambda}{\mu} t + \frac{1}{n^2}   r_1^2 t + O \big ( \frac{1}{n^4} \big )  \right \} \, ,
$$
gives rise using \eqref{osc1} to oscillating solutions for \eqref{lintve1} which read
\begin{equation}
\label{oscsol1}
\begin{bmatrix}
u \\[5pt]  \frac{\del u}{\del t}  \\[5pt] \theta
\end{bmatrix} (t,x)
=
\begin{bmatrix} \tfrac{1}{n} a(t)  \\  \tfrac{1}{n} v(t)   \\ - i b(t) \end{bmatrix} 
e^{i n x}  
=
\begin{bmatrix}   \frac{1}{n} \mu - \frac{\lambda}{\mu} \frac{1}{n^3}  +  O ( \frac{1}{n^5} )  \\[5pt] -\frac{1}{n} \lambda +  \frac{m^2 \lambda}{\kappa \mu} \frac{1}{n^3} +  O ( \frac{1}{n^5} ) \\[5pt]  
- i  \frac{m \lambda}{\kappa} \frac{1}{n^2} + O ( \frac{1}{n^4} ) \end{bmatrix} 
\exp \Big \{ i n x -\frac{\lambda}{\mu} t +O \big (  \frac{1}{n^2} \Big )  t )  \Big \}
\end{equation}
Considering \eqref{oscsol1} we see that $u$, $\frac{\del u}{\del t}$, $\theta$ and $\frac{\del \theta}{\del x}$ converge to zero for $n$ large, but oscillations persist
for the strain $w = \frac{\del u}{\del x}$, the strain rate $\frac{\del v}{\del x}$ and $\frac{\del^2 \theta}{\del x^2}$. To leading order these solutions read
\begin{equation}
\label{oscsol2}
\begin{bmatrix}
w  \\[5pt]  \frac{\del v}{ \del x}  \\[5pt] \frac{\del^2 \theta}{\del x^2}
\end{bmatrix} (t,x) 
=
\begin{bmatrix}
\frac{\del u}{\del x}  \\[5pt]  \frac{\del^2 u}{\del t \del x}  \\[5pt] \frac{\del^2 \theta}{\del x^2}
\end{bmatrix} (t,x) 
=
i \, \begin{bmatrix}   \mu + O ( \frac{1}{n^2} )  \\[5pt] -  \lambda +   O ( \frac{1}{n^2} ) \\[5pt]  
  \frac{m \lambda}{\kappa} + O ( \frac{1}{n^2} ) \end{bmatrix} 
\left ( 1 +O \big (  \frac{1}{n^2}   \big ) t  \right ) \exp \Big \{ i n x -\frac{\lambda}{\mu} t  \Big \}
\end{equation}

\begin{remark}
A special case of \eqref{linate} is the system of adiabatic 1-$d$ linear thermoviscoelasticity when $\kappa = 0$,
\begin{equation}\label{lintvead2}
\begin{aligned}
\frac{\del w}{\del t}   &= \; \frac{\del v}{\del x}
\\
\frac{\del v}{\del t} &= \lambda \frac{\del w}{\del x} + m \frac{\del \theta}{\del x} + \mu \frac{\del^2 v }{\del x^2} 
\\
\frac{\del \theta}{\del t} &=  m \frac{\del v}{\del x } 
\end{aligned}
\end{equation}
We will see in section \ref{sec:ex2} that one may construct solutions of \eqref{lintvead2} for which 
oscillations appear to leading order for the vector function $(w, v_x, \theta)$. This should be contrasted to the behavior
of \eqref{lintve2} (which has $\kappa \ne 0$) where oscillations appear to  leading order for the vector function $(w, v_x, \theta_{xx})$, see \eqref{oscsol1}, \eqref{oscsol2}.
\end{remark}

\medskip
\subsection{An illustrating example} 
\label{sec:ex1d}
The special case of linear viscoelasticity (in one and several space dimensions) presents an example
that can be worked out explicitly. For the equation
\begin{equation}
\label{linearvisc}
u_{tt} = \lambda u_{xx} + \mu u_{t x x}  \qquad x \in (-\pi, \pi) \, , t > 0 \, ,
\end{equation}
where $u : (0,T)\times \Tone \to \R$,  we consider solutions of the form $u(t,x) = \frac{1}{n} \alpha(t) e^{ i n x}$.
(For the multi-dimensional case, see section \ref{sec:ex2}).
The amplitude $\alpha (t)$ satisfies
\begin{equation}
\label{exode1}
\frac{d^2\alpha}{d t^2} + \mu n^2 \frac{d \alpha}{d t} + n^2 \lambda \alpha = 0 \, .
\end{equation}
The associated characteristic polynomial  $\rho^2 + \mu n^2 + \lambda n^2 = 0$ has roots
$$
\rho_\pm = \frac{\mu n^2}{2} \left ( - 1 \pm \sqrt{ 1 - \frac{4 \lambda}{\mu^2 n^2}} \right )
$$
which are real for $n$ large, and enjoy the asymptotic expansions
$$
\begin{aligned}
\rho_-  &= - \mu n^2 + \frac{\lambda}{\mu} + \frac{\lambda^2}{\mu^3 n^2} + O \big ( \tfrac{1}{n^4} \big)
\\
\rho_+  &= - \frac{\lambda}{\mu} - \frac{\lambda^2}{\mu^3 n^2} + O \big ( \tfrac{1}{n^4} \big)
\end{aligned}
$$
They give rise to solutions $\alpha_\pm = e^{\rho_\pm t}$ for \eqref{exode1} and $u_{n , \pm} = \frac{1}{n}  e^{ i n x + \rho_\pm t}$ for \eqref{linearvisc}.
Of interest here is the solution associated to $\rho_+$ and we consider $(u_n, v_n)$ where 
\begin{equation}
\label{beh1}
u_{n}  (t,x) =  \frac{1}{n} \exp \left \{   i n x   - \frac{\lambda}{\mu} t  - \frac{\lambda^2}{\mu^3 n^2} t + O \big ( \tfrac{1}{n^4}   \big)  t   \right \}
\end{equation}
while $v_n = \frac{\del u_n }{\del t}$.   Note that $u_n \, ,  v_n \to 0$ as $n \to \infty$.  By contrast,
\begin{equation}
\label{beh2}
\begin{aligned}
\frac{\del u_n}{\del x}  &= i  \left ( 1 + O \big ( \frac{1}{n^2} \big ) \right ) \exp \left \{   i n x   - \frac{\lambda}{\mu} t   \right \}
\\
\frac{\del v_n}{\del x} &=  i  \left (  - \frac{\lambda}{\mu}  + O \big ( \frac{1}{n^2} \big ) \right ) \exp \left \{   i n x   - \frac{\lambda}{\mu} t   \right \}
\end{aligned}
\end{equation}
exhibit  persistent oscillations induced by the initial data.

To put this behavior into perspective, note that \eqref{linearvisc} can be expressed as the hyperbolic-parabolic system,
\begin{equation}
\label{exode2}
\begin{aligned}
\frac{\del w}{\del t}   &= \; \frac{\del v}{\del x}
\\
\frac{\del v}{\del t} &= \lambda \frac{\del w}{\del x} +  \mu \frac{\del^2 v }{\del x^2} 
\end{aligned}
\end{equation}
with $w = \frac{\del u}{\del x}$. 
The oscillation in \eqref{beh2} reflects a balance between the terms $w$ and $\frac{\del v}{\del x}$
and is stationary in time.

It should be contrasted with the usual oscillatory solutions of hyperbolic systems. Recall, that the hyperbolic model
\begin{equation}
\label{exode4}
\begin{aligned}
\frac{\del w}{\del t}   &= \; \frac{\del v}{\del x}
\\[3pt]
\frac{\del v}{\del t} &= \lambda \frac{\del w}{\del x} 
\end{aligned}
\end{equation}
has oscillatory solutions propagating in the directions of the wave speeds
$$
\begin{bmatrix}
w  \\[5pt]  v  
\end{bmatrix} (t,x) 
= \xi_\pm  \exp  \big \{  i n (x - \pm{\sqrt{\lambda}} t  )  \big \}
$$
where $\xi_\pm$ are the eigenvectors associated to the wave speeds $A \xi_\pm = \pm \sqrt{\lambda} \xi_\pm$. These propagate 
with speeds $\pm \sqrt{\lambda}$ and are of different nature than \eqref{beh1}, \eqref{beh2} which are  stationary in time.

\medskip
\subsection{Oscillations in linear multi-dimensional systems} \label{sec:oscmd}
Consider next the multi-dimensional system
\begin{equation}
\label{lintved}
\begin{aligned}
\frac{\del^2 u_k}{\del t^2} &= \frac{\del}{\del x_\alpha} \left ( A_{k l \alpha \beta} \frac{\del u_l }{\del x_\beta} + M_{k \alpha} \theta \right ) 
+ \mu \Delta \left ( \frac{\del u_k}{\del t} \right )
\\[5pt]
\frac{\del \theta}{\del t} &= \kappa  \Delta \theta + M_{k \alpha} \frac{\del}{\del x_\alpha} \left ( \frac{\del u_k}{\del t} \right )
\end{aligned}
\end{equation}
describing the evolution of a thermomechanical process $(u, \theta) : (0,\infty)\times \R^d \to \R^d \times \R^+$. Here, the summation convention is used, 
the parameters $\kappa > 0, \mu > 0$ describe the viscosity and thermal diffusivity, $A_{k l \alpha \beta}$ is a fourth order rank-1 convex tensor which
is symmetric in the sense $A_{k l  \alpha \beta} = A_{l k  \alpha \beta}$, $A_{k l \alpha \beta} = A_{k  l \beta \alpha}$ and describes the elastic response, 
while $M_{k \alpha}$ is a second order tensor describing the thermoelastic coupling.

Consider the {\it ansatz} of oscillating solutions of the form
\begin{equation}
\label{oscd}
\begin{aligned}
u_k (t, x) &= \frac{1}{n} a_k (t) e^{ i n \nu \cdot x}   \qquad k = 1, ... , d
\\
\theta(t,x) &= - i b (t) e^{ i  n \nu \cdot x}
\end{aligned}
\end{equation}
where $\nu = (\nu_1, ... , \nu_d)$ is a unit vector,  $| \nu | =1$, and n is a natural number. Introducing the ansatz \eqref{oscd}  to \eqref{lintved} we see
that  the amplitudes $a_k (t)$, $\beta (t)$ satisfy the system of ordinary differential equations
\begin{equation}
\label{oded}
\begin{aligned}
\ddot{a}_k  &= - n^2  (A_{k l \alpha \beta} \nu_\alpha \nu_\beta)  a_l + n^2 (M_{k \alpha} \nu_\alpha ) b- \mu n^2 \dot{a}_k
\\
\dot{b} &= - n^2 \kappa b - ( M_{k \alpha} \nu_\alpha ) \dot{a}_k
\end{aligned}
\end{equation}
Note that \eqref{oded} admits real valued solutions $(a_1, ... , a_d , b)$ which by taking real and imaginary parts in \eqref{oscd} give rise to 
real-valued oscillating solutions for \eqref{lintved}.

Recalling the summation convention is used, we proceed to solve the system \eqref{oded}. As already noted,
the fourth order tensor $A_{k l  \alpha \beta}$ is assumed symmetric and rank-1 convex:
\begin{align}
\label{hypsym}
A_{k l  \alpha \beta} = A_{l k  \alpha \beta} \,  , \quad &A_{k l \alpha \beta} = A_{k  l \beta \alpha} \, ,
\tag{H$_1$}
\\
\label{rank1cvx}
A_{k l  \alpha \beta}  \nu_\alpha \nu_\beta \xi_k \xi_l > 0 \quad &\forall \xi \in \R^d - \{0\} \, ,  \nu \in S^{d-1} \, .
\tag{H$_2$}
\end{align}
Hypothesis \eqref{rank1cvx} implies the acoustic tensor $Q_{k l} := A_{k l  \alpha \beta}  \nu_\alpha \nu_\beta $ is symmetric, 
positive definite. It thus has positive eigenvalues $\lambda^r > 0$, and a complete set of linearly independent eigenvectors $\xi^r$,
$r = 1, ... , d$, satisfying
\begin{equation}
\label{eigenv}
Q_{k l} \xi_l^r = \lambda^r \xi_k^r  \quad k = 1, ... , d \, \; \; \mbox{ normalized so that $|\xi^r| = 1$. }
\end{equation}
We now place an assumption on the thermoelastic interaction matrix $M_{k \alpha}$. Given $\xi^r$ we assume there is $\nu \in S^{d-1}$ 
such that for some $m \in \R$
\begin{equation}
\label{hypm}
M_{k \alpha } \nu_\alpha =  m \xi^r_k   \quad k = 1, ... d .
\tag{H$_3$}
\end{equation}
Note that $|\xi^r|=1$ implies $m^r = \xi_k M_{k \alpha} \nu_\alpha$.
Hypothesis \eqref{hypm} is automatically satisfied if $\text{rank} M = d$. Otherwise, it is an assumption connecting
pairs of eigenvectors $\xi^r$ to  associated propagation directions $\nu$.

With these in place, fix $r = 1, ... , d$ and proceed to solve \eqref{oded}. We make the ansatz
\begin{equation}
\label{anform}
a_k (t) = \alpha(t) \xi_k^r  \quad \mbox{ where  $\xi^r \in \R^d$, $|\xi^r | = 1$ and $\alpha(t)$ scalar valued}.
\end{equation}
Then  $(\alpha (t), b (t))$ satisfies the system of ordinary differential equations
\begin{equation}
\label{ode3}
\begin{aligned}
\ddot{\alpha} &=  - n^2 \lambda^r \alpha + n^2 m^r b - \mu n^2 \dot{\alpha}
\\
\dot{b} &= - \kappa n^2 b - m^r \dot{\alpha}
\end{aligned}
\end{equation}
where $m^r = \xi^r \cdot M \nu$, and $\lambda^r$,  $\xi^r$ are an eigenvalue and corresponding eigenvector of the acoustic tensor.
Conversely, \eqref{hypm}, \eqref{eigenv} imply solutions of \eqref{ode3} generate via \eqref{anform}  solutions of \eqref{oded}.

The system \eqref{ode3} is precisely \eqref{ode2} studied  in section \ref{sec:osc1d}.
Combining the analysis of subsection \ref{sec:osc1d} with \eqref{anform} we conclude that if \eqref{hypsym}, \eqref{rank1cvx} and \eqref{hypm} are satisfied 
for an eigenpair $(\lambda^r, \xi^r)$ and associated direction $\nu$, there will be oscillating solutions for \eqref{lintved}
of the form
\begin{equation}
\label{oscsold}
\begin{aligned}
\begin{bmatrix}
u \\[5pt]  \frac{\del u}{\del t}  \\[5pt] \theta
\end{bmatrix} (t,x)
&=
\begin{bmatrix} \tfrac{1}{n} \alpha(t) \xi^r  \\[5pt]  \tfrac{1}{n} \dot{\alpha} (t) \xi^r   \\[5pt]  - i b (t) \end{bmatrix} 
e^{i n \nu \cdot x}  
\\
&=
\begin{bmatrix}   \left ( \frac{1}{n} \mu + O \big ( \frac{1}{n^3}  \big ) \right ) \xi^r  \\[8pt] 
\left ( -\frac{1}{n} \lambda^r +  O ( \frac{1}{n^3} )  \right ) \xi^r \\[8pt]  
- i \frac{1}{n^2} \frac{m^r  \lambda^r }{\kappa} + O ( \frac{1}{n^4} ) \end{bmatrix} 
\exp \Big \{ i n \nu \cdot x -\frac{\lambda^r}{\mu} t +O \big (  \frac{1}{n^2} \Big )  t )  \Big \}
\end{aligned}
\end{equation}
where $(\alpha(t), b(t))$ satisfy \eqref{ode3}.
These provide oscillating progressive waves in the direction $\nu$. Again $u$, $v = \frac{\del u}{\del t}$
and $\theta$ converge to zero, $\nabla u$, $\nabla v$ and $\Delta \theta$ have persistent oscillations
\begin{equation}
\label{behexa2}
\begin{aligned}
\nabla u &=  i \mu  ( \xi^r \otimes \nu )  \left ( 1 +  O \big (  \frac{1}{n^2} \Big )  t \right ) \exp \Big \{ i n \nu \cdot x -\frac{\lambda^r}{\mu} t  \Big \}
\\
\Delta \theta &=  i  \frac{m^r  \lambda^r }{\kappa}   \left ( 1 +  O \big (  \frac{1}{n^2} \Big )  t \right ) \exp \Big \{ i n \nu \cdot x -\frac{\lambda^r}{\mu} t  \Big \}
\end{aligned}
\end{equation}
induced by oscillations in the initial data.

\subsection{Linear adiabatic thermoviscoelasticity}
\label{sec:ex2}
Another system where oscillatory solutions may be constructed explicitly is the case of adiabatic thermoviscoelastcity, namely
the special case of \eqref{lintved} when $\kappa = 0$, which reads
\begin{equation}
\label{linate}
\begin{aligned}
\frac{\del^2 u_k}{\del t^2} &= \frac{\del}{\del x_\alpha} \left ( A_{k l \alpha \beta} \frac{\del u_l }{\del x_\beta} + M_{k \alpha} \theta \right ) 
+ \mu \frac{\del}{\del x_\alpha} \frac{\del}{\del x_\alpha} \left ( \frac{\del u_k}{\del t} \right )
\\[5pt]
\frac{\del \theta}{\del t} &=  M_{k \alpha} \frac{\del}{\del x_\alpha} \left ( \frac{\del u_k}{\del t} \right )
\end{aligned}
\end{equation}
Introducing \eqref{oscd} leads to $(a_1, ... a_d , b)$ solving the system of ordinary differential equations \eqref{oded}
with $\kappa = 0$, To this end, we look for $a(t) = (a_1, ... a_d ) (t)$ solving 
\begin{equation}
\label{odeds}
\begin{aligned}
\ddot{a}_k  &= - n^2  (A_{k l \alpha \beta} \nu_\alpha \nu_\beta)  a_l -  n^2 (M_{k \alpha} \nu_\alpha ) ( M_{l \beta} \nu_\beta ) a_l  - \mu n^2 \dot{a}_k
\end{aligned}
\end{equation}
and then select
\begin{equation}
\label{eqb}
\begin{aligned}
b = - ( M_{k \alpha} \nu_\alpha ) a_k \, .
\end{aligned}
\end{equation}
The resulting $(a, b)$ will produce  via \eqref{oscd} solutions of \eqref{linate}.

We impose the hypotheses \eqref{hypsym}, \eqref{rank1cvx} of rank-1 convexity for $A_{k l \alpha \beta}$ (but not hypothesis \eqref{hypm}).
Fix a direction $\nu \in S^{d-1}$ and consider the modified acoustic tensor
$$
Q^m_{kl} = A_{k l \alpha \beta} \nu_\alpha \nu_\beta + (M_{k \alpha} \nu_\alpha ) ( M_{l \beta} \nu_\beta )
$$
By \eqref{rank1cvx} the matrix $Q^m_{kl}$ is positive definite. Let $\sigma^{ r} > 0$ be an eigenvalue and $\zeta^r$ the corresponding eigenvector,
$$
Q^m \zeta^r = \big ( Q + ( M \nu ) \otimes (M \nu) \big ) \zeta^r = \sigma^r \zeta^r \, , \quad | \zeta^r| =1  \, .
$$
Set $a_k = \alpha (t) \zeta^r_k$, then \eqref{odeds} is equivalent to solving the equation for the amplitude $\alpha (t)$
\begin{equation}
\label{exode3}
\frac{d^2\alpha}{d t^2} + \mu n^2 \frac{d \alpha}{\del t} + n^2 \sigma^r \alpha = 0 \, ,
\end{equation}
which is precisely the example studied in section \ref{sec:ex1d}.

The analysis in section \ref{sec:ex1d} indicates there are two solutions of \eqref{exode3} one that decays fast, and
a second denoted there by  $\alpha_+ (t) = e^{\rho_+ t}$ which decays slowly. Utilizing \eqref{oscd}, \eqref{eqb} the slowly decaying solution we conclude that \eqref{linate}
has solutions of the form
\begin{align}
u_k = \frac{1}{n} \alpha (t) \zeta_k^r e^{i n \nu \cdot x}  &= \frac{1}{n}  \zeta_k^r  \exp \left \{ i n \nu \cdot x - \frac{\sigma^r}{\mu}t + O \big ( \frac{1}{n^2} \big ) t \right \}
\\
\nonumber
\theta = i ( M_{k \alpha} \nu_\alpha  \zeta^r_k ) \alpha (t) e^{i n \nu \cdot x}  &= i ( M_{k \alpha} \nu_\alpha  \zeta^r_k ) \exp \left \{ i n \nu \cdot x - \frac{\sigma^r}{\mu}t + O \big ( \frac{1}{n^2} \big ) t \right \}
\\
\label{behtheta}
&= i (\zeta^r \cdot M \nu ) \Big ( 1 + O \big ( \frac{1}{n^2} \big )  t \Big )  \exp  \left \{ i n \nu \cdot x - \frac{\sigma^r}{\mu}t  \right \}
\end{align}
Note that  $u_n , v_n \to 0$ as $n \to \infty$ but oscillations persist for $\theta$,  
\begin{equation}
F = \nabla u = i ( \zeta^r \otimes \nu)   \Big ( 1 + O \big ( \frac{1}{n^2} \big )  t \Big )   \,  \exp  \left \{ i n \nu \cdot x - \frac{\sigma^r}{\mu}t  \right \}
\end{equation}
and  $\nabla v$.

In the above example there are oscillations of $\theta$ to leading order, see \eqref{behtheta}. This behavior should be contrasted to the example of section 
\ref{sec:oscmd} where, due to the presence of heat diffusion, oscillations first appear in $\Delta \theta$, see \eqref{behexa2}.


\section{Oscillations in quasi-linear dynamical models for phase transitions}\label{sec:sust}

We record some examples of solutions exhibiting sustained oscillations for various
quasi-linear hyperbolic-parabolic systems. The first class of models concerns viscoelastic systems of
the type used in modeling of phase transitions.

\subsection{ Time-dependent twinning solutions in nonlinear viscoelasticity}\label{sec:twin}

In preparation, we mention a well known result from continuum mechanics. Let $\psi(x)$ be a displacement field defined on
a reference configuration $\cR$ and suppose that $\cR$ is split into two subdomains, $\cR = \cR^+ \cup \cR^-$, by a smooth 
nonsingular hypersurface $\cS$ defined by
$$
f(x) = 0  \quad \mbox{with $\nabla f \ne 0$}.
$$
Suppose that $\psi$ enjoys the regularity $\psi$ is continuous on $\cR$, $\psi \in C^1\big(\overline{\cR^-}\big)$, $\psi \in C^1\big( \overline{\cR^+}\big)$,
and the limits
$$
\lim_{ \scriptsize\begin{matrix} x\to x_0 \\ x> x_0 \end{matrix}} \nabla \psi(x) = F^+ \, , \quad \lim_{\scriptsize \begin{matrix} x\to x_0 \\ x< x_0 \end{matrix} }\nabla \psi(x) = F^- \, ,
$$
exist for $x_0 \in \cS$ and are finite. Then the deformation gradients satisfy across $\cS$ the jump conditions
$$
F^+ - F^- = a \otimes N
$$
where $N$ is the normal to the surface and $a$ is an amplitude capturing the jump of the gradient in the normal direction at $\cS$. This result is the backbone 
in the  construction of twinning solutions in elasticity, \cite{James81, BJ87}.

Consider now the system of viscoelasticity of Kelvin-Voigt type,
\begin{align}\label{vekv}
&\partial_{tt}y-\dive  \Big (  \frac{\del W}{\del F} (\nabla y) + \nabla y_t  \Big ) = 0 \, .
\end{align}
It is written as a system of conservation laws expressed in coordinate form, for the quantities 
$v_i = \frac{\del y_i}{\del t}$, $F_{i \alpha} = \frac{\del y_i}{\del x_\alpha}$, by
\begin{equation}\label{vekv-coord}
\begin{aligned}
\del_t F_{i \alpha} &= \del_\alpha v_i
\\
\del_t v_i &= \del_\alpha \Big ( \frac{\del W}{\del F_{i \alpha}} (F) + \del_\alpha v_i \Big )
\\[3pt]
0 &= \del_\beta F_{i \alpha } - \del_\alpha F_{i \beta} \, .
\end{aligned}
\end{equation}
The last equation constrains $F$ to be a gradient. It is an involution, i.e. a constraint that propagates
from the initial data to solutions. 

Our purpose is to construct oscillating solutions. This will result from two ingredients, a special class of solutions
combined with conditions that guarantee joining these solutions across interfaces. We start with the interface conditions.

Suppose that $\cS = \{ (t, x) : f(x,t) = 0\}$ is a space-time hypersurface that splits a space time domain $Q_T = \cR \times (0,T)$ into two parts, 
$Q_T = \cR^+ \cup \cR^-$. We assume that $\cS$ is smooth and depicts the motion of an interface $x = x(t)$ which is moving in the direction of the normal
$\nu = \frac{\nabla_x f}{| \nabla_x f |}$ with speed $\dot x = s \nu $.  Since the interface $x(t)$ moves on the surface, we have $f(x(t), t) = 0$, which easily
yields for the interfacial speed that $s = - \frac{f_t}{| \nabla_x f|}$.  Suppose a solution $(v, F)$ of \eqref{vekv} 
that has smoothness $v \in C(Q_T)$, and such that $F, \nabla v \in  C \big(\overline{\cR^-}\big)$,  $F, \nabla v \in C \big( \overline{\cR^+}\big)$
and the right and left limits of $F, \nabla v$ exist and are finite on both sides of the space-time surface $\cS$.  The limits will satisfy the Rankine-Hugoniot conditions
\begin{align}
- s [ F_{i \alpha}] &= \nu_\alpha [v_i]
\label{rh1}
\\
- s [v_i] &= \nu_\alpha \left [ \frac{\del W}{\del F_{i \alpha}} (F) + \del_\alpha v_i \right ]
\label{rh2}
\\[3pt]
0 &= \nu_\beta [F_{i \alpha } ] - \nu_\alpha [F_{i \beta}]
\label{rh3}
\end{align}
where $[ \cdot ]$ denotes the jump across the interface.
Condition \eqref{rh3} implies that $[F] \tau = 0 $ for any vector $\tau \perp \nu$ and thus $[F] = a \otimes \nu$. Together with \eqref{rh1} they 
lead to
\begin{equation}\label{kinemjump}
[F] = a \otimes \nu \, , \quad [v] = - s a \, .
\end{equation}
We are here interested in $[F] \ne 0$ and $[v] = 0$. 
The latter precludes $\nabla v$ from having delta masses and allows the jump to be computed via \eqref{rh2}.
This suggests to consider standing shocks $s=0$ in conformance with the considerations of the one-dimensional case studied in \cite{Hoff86}.
In summary,  $s=0$ and $[F] \ne 0$ have to satisfy
\begin{equation}\label{steadysh}
[F] = a \otimes \nu \, , \quad   \big [  T \nu \big ] = \Big [  \Big ( \frac{\del W}{\del F} (F) + \nabla v \Big ) \nu \Big ] = 0
\end{equation}
where $T =  \frac{\del W}{\del F} (F) + \nabla v $ is the total stress. Observe that \eqref{steadysh} requires that the steady interface is in equilibrium.

The second ingredient is special solutions of \eqref{vekv} of the particular form $y(t,x) = F(t) x$ with $F(t)\in \R^{d \times d} $ a time dependent matrix.
One checks that $F(t)$ satisfies $\ddot F = 0$, which implies that the solution of \eqref{vekv} must be of the form
\begin{equation}\label{usmd}
y(t,x) = ( t F_0 +  F_1) x  \, , \quad F_0, F_1 \in \R^{d \times d} \, ,
\end{equation}
where $F_0$, $F_1$ are constant matrices. They describe uniform shear and are
universal in the sense that they are independent of the form of the stored energy $W(F)$.

We use \eqref{kinemjump} and \eqref{steadysh} to construct some special solutions of \eqref{vekv-coord}, which are in turn used to construct 
oscillatory solutions for the same system.

\subsubsection{Steady solutions with jumps in the deformation gradient in elasticity} \label{sec:steadyjump}
The first solution is known from work of Ball and James \cite{BJ87} and
provides solutions with discontinuities in the deformation gradient for the equations of elasticity. The solutions are steady and solve \eqref{vekv-coord}.
Namely, let $F_-$, $ F_+$ be such that $F_+ - F_- = a \otimes \nu$ with $a \in \R^d$, $\nu \in \mathcal{S}^{d-1}$. The hyperplane $\nu \cdot x =0$ passes 
through the origin and is orthogonal to the unit vector $\nu$.  The function
\begin{equation}
y (t, x) = 
\begin{cases} 
F_-  x  & x \cdot \nu < 0 \\
F_+ x &  x \cdot \nu > 0 \\
\end{cases}
\end{equation}
satisfies $v(t,x) = 0$ and is a steady solution of 
\begin{equation}\label{vekv-coord-el}
\begin{aligned}
\del_t F_{i \alpha} &= \del_\alpha v_i
\\
\del_t v_i &= \del_\alpha \Big ( \frac{\del W}{\del F_{i \alpha}} (F) \Big )
\\[3pt]
0 &= \del_\beta F_{i \alpha } - \del_\alpha F_{i \beta}
\end{aligned}
\end{equation}
and also of \eqref{vekv-coord} provided
$$
F_+ - F_- = a \otimes \nu  \, , \qquad 
[  T \nu \big ] = \Big (  \frac{\del W}{\del F} (F_+) - \frac{\del W}{\del F} (F_-) \Big )   \nu = 0 \, .
$$
This solution  has discontinuous deformation gradient across an interface, and  is related to twinning \cite{James81}. The position of
the matrices $F_-$, $F_+$ can be interchanged and leads ato a different weak solution. Combining these together can lead to a persistent steady
oscillatory structure that is related to phase transitions, \cite{BJ87}.

\subsubsection{Dynamic solutions with jumps in deformation gradient and strain rate in viscoelasticity}\label{sec:vejump}
Next consider $F_0$ a baseline deformation gradient,  $F_- = F_0 + a \otimes \nu$, $F_+ = F_0 + b \otimes \nu$,
where $a$, $b \in \R^d$, $\nu \in \mathcal{S}^{d-1}$. Then $F_+ - F_- = (b-a) \otimes \nu$. The hyperplane $x \cdot \nu = 0$ passes through the origin.
Consider the motion
\begin{equation}\label{basicsol}
\begin{aligned}
y (t, x) &= 
\begin{cases} 
y_- (t,x)  =  t \big ( F_0 + a \otimes \nu \big ) x   & \quad  x \cdot \nu < 0 \\
y_+ (t,x)  =  t \big ( F_0 + b \otimes \nu \big ) x  &  \quad x \cdot \nu > 0 \\
\end{cases}
\\[5pt]
v (t, x) &= 
\begin{cases} 
 \big ( F_0 + a \otimes \nu \big ) x   &\quad  x \cdot \nu < 0 \\
 \big ( F_0 + b \otimes \nu \big ) x  &  \quad x \cdot \nu > 0 \\
\end{cases}
\end{aligned}
\end{equation}
Note that $y$ and $v$ are both continuous across the interface $x \cdot \nu = 0$  which is steady in time.  The deformation gradient $\nabla y$ and stretch tensor $\nabla v$
have discontinuities across that interface.  Moreover,  by virtue of \eqref{steadysh},  $y$ will be a weak solution of  \eqref{vekv-coord} on the domain 
$[1,2] \times \R^d$ provided
\begin{equation}\label{cond}
\Big (  \frac{\del W}{\del F} ( t F_- + t (b-a) \otimes \nu ) - \frac{\del W}{\del F} ( t F_-)  + (b - a) \otimes \nu \Big )   \nu = 0 \, ,
\quad \mbox{ for $t \in [1,2]$},
\tag{C}
\end{equation}
is satisfied. The reader should note that under \eqref{cond}  the role of $a$ and $b$ (or $F_-$ and $F_+$) can be interchanged 
so as to obtain a second solution with discontinuities in the deformation gradient $\nabla y$ and stretching $\nabla v$.
\begin{figure}[htbp] 
\begin{tikzpicture}[scale=1.8]
\def\ang{-10} 
\def\th{.4} 

\draw[->] (-2,0) -- (3.5,0) node[below, font=\normalsize] {$x_1$}; 
\draw[->] (0,-2) -- (0,3) node[right, font=\normalsize] {$x_2$}; 

\begin{scope}[rotate=\ang] 

\foreach \x in {-2,-1,...,2} { 
  \draw (\x,-1.5) -- (\x,0) node[point]{} -- (\x,2.5) 
    (\x+\th,-1.5) -- (\x+\th,0) node[point]{} -- (\x+\th,2.5); 
  \path (\x+\th/2,1) node{$tF_-$}  (\x+1/2+\th/2,1) node{$tF_+$}; 
  \path[every node/.style={below, fill=white, inner sep=2pt, outer sep=3pt, fill opacity=.85, text opacity=1.}] 
    (\x,0) node[rotate=\ang]{$\x\mathstrut$ } 
    (\x+\th,0) node[rotate=\ang]{\ifnum \x=0 $\theta\mathstrut$ \else $\x{+}\theta\mathstrut$ \fi}; 
}; 
\draw (3,-1.5) -- (3,0) node[point]{} 
  node[below, fill=white, inner sep=2pt, outer sep=3pt, fill opacity=.85, text opacity=1., rotate=\ang]{$3\mathstrut$} -- (3,2.5); 

\path[above] 
(0,2.5) node[yshift=1mm]{$x\cdot\nu=0$} 
(\th,2.5) node{$x\cdot\nu=\theta$} 
(1,2.5) node{$x\cdot\nu=1$} 
(2,2.5) node{$x\cdot\nu=2$}; 

\draw[->, very thick] (0,0) -- (1,0) node[above left]{$\nu$}; 
\draw[dashed] (-2.25,0) -- (3.5,0); 

\end{scope} 
\end{tikzpicture}  
\caption{Deformation gradient for oscillating solutions, $F_+ = F_0+a\otimes\nu$, $F_- = F_0+b\otimes\nu$.} \label{Figoscillation}
\end{figure}
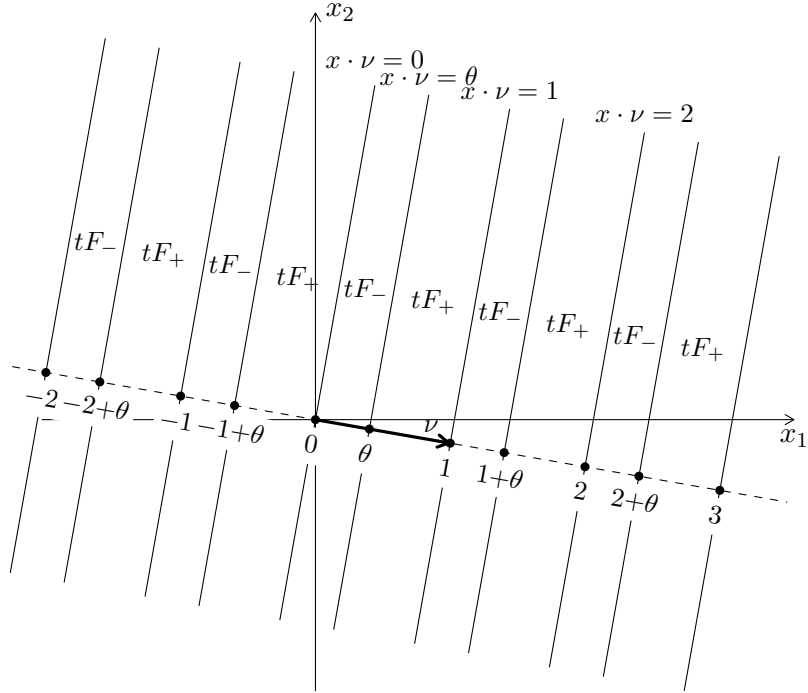 

Suppose that the condition \eqref{cond} is satisfied. By interlacing the solution \eqref{basicsol} with the corresponding solution with the position
of $F_-$ and $F_+$ interchanged, we construct an oscillating solution defined on $[1,2] \times \R^d$ that jumps across the steady interfaces 
$x \cdot \nu = k$ and $x \cdot \nu = k + \theta$, $k \in \mathbb{Z}$. The deformation gradient and stretching are given respectively by
\begin{equation}
\begin{aligned}
\nabla y (t,x) &= \begin{cases} t  ( F_0 + a \otimes \nu )   & \quad   k< x \cdot \nu < k + \theta  \\
                                              t ( F_0 + b \otimes \nu  )   &  \quad  k + \theta < x \cdot \nu < k + 1  \\
\end{cases}   \qquad k \in \mathbb{Z}
\\[5pt]
\nabla v (t,x) &= \begin{cases}  F_0 + a \otimes \nu    & \quad   k< x \cdot \nu < k + \theta  \\
                                                F_0 + b \otimes \nu   &  \quad  k + \theta < x \cdot \nu < k + 1  \\
\end{cases}   \qquad k \in \mathbb{Z}
\end{aligned}
\end{equation}
and the deformation gradient is presented in Figure \ref{Figoscillation}.

A commonly used notion in elasticity theory is rank-one convexity. Recall that $W(F)$ is rank-one convex on the domain $D$ if for any
$F \in D$ 
\begin{equation}\label{ROC}
\sum_{i,j, \alpha, \beta} \frac{\del^2 W}{\del F_{i \alpha} \del F_{j \beta}} (F) \;  \nu_\alpha \nu_\beta \; \xi_i \xi_j  > 0 \, ,
\quad   \xi \in \R^d  - \{0\} \, , \, \nu \in \mathcal{S}^{d-1}  \, .
\tag {ROC}
\end{equation}
Observe that :

\begin{lemma}
If condition \eqref{cond} is satisfied then \eqref{ROC} is violated
\end{lemma}

\begin{proof}
Indeed, using the summation convention, we express \eqref{cond} as
$$
\begin{aligned}
0 &= \Big (  \frac{\del W}{\del F_{i \alpha}} ( t F_- + t (b-a) \otimes \nu ) - \frac{\del W}{\del F_{i \alpha}} ( t F_-)  + (b - a)_i \nu_\alpha  \Big )   \nu_\alpha 
\\
&=  \left ( \int_0^1 \frac{d}{ds} \frac{\del W}{\del F_{i \alpha}} ( t F_- + s  t (b-a) \otimes \nu ) \, ds \right ) \nu_\alpha  + (b-a)_i
\\
&= t  \left ( \int_0^1  \frac{\del^2 W}{\del F_{i \alpha} \del F_{j \beta} } ( t F_- + s  t (b-a) \otimes \nu ) \, ds \right ) (b-a)_j \nu_\alpha \nu_\beta + (b-a)_i
\end{aligned}
$$
IF $W(F)$ is rank-one convex on the domain $D$ containing $t F_- + s  t (b-a) \otimes \nu$  then  we obtain $0 > |b - a|^2$ leading to a contradiction.
\end{proof}

\subsection{Longitudinal motions of a viscoelastic bar}\label{sec:1dnonmon}

We outline a one-dimensional example which is a variant of the example in \cite[sec 7]{KLST23}
and concerns the system describing one-dimensional longitudinal motions of a viscoelastic bar.
The system 
\begin{equation}
\label{longonev}
\begin{aligned}
u_t &= v_x
\\
v_t &= \sigma(u)_x +  \del_x \left ( \frac{\mu}{u} v_{x} \right )
\end{aligned}
\end{equation}
describes longitudinal motions  $y(t,x) :(0,T) \times [0,1] \to \R$ for a bar, where $u = y_x > 0$ is
the longitudinal strain and $v= y_t$ the velocity. The total stress 
$$
S = \sigma (y_x) +    \frac{\mu }{y_x} y_{t x}
$$ 
has an elastic and a viscous component, with the viscosity in Lagrangian coordinates expressed as $\frac{\mu}{u}$.
Here we take $\mu = 1$ and the function $\sigma (u)$ is smooth and  non-monotone. Such models have
been used for modeling phase transitions.

Assume two positive states $0 < a < b$ are fixed and suppose the stress function $\sigma (u)$
satisfies
\begin{equation}
\label{condnonm}
 \sigma (\tau  a) = \sigma ( \tau  b) \qquad \mbox{ for $\tau \in [1,2]$} \, .
\end{equation}
Clearly, \eqref{condnonm} requires that $\sigma(u)$ is nonmonotone.
The example detailed below is based on two properties of the system \eqref{longonev}:
\begin{itemize}
\item[(i)]  It admits special solutions of the form
\begin{equation}\label{unifshear}
\bar u (t) = \kappa t \, , \quad \bar v (x) = \kappa x
\end{equation}
where $\kappa > 0$ is a rate of stretching. These solutions are universal in the sense that they satisfy
\eqref{longonev} for any function $\sigma (u)$.

\item[(ii)]  \eqref{longonev} admits piecewise smooth solutions that are 
continuous in $v$ but discontinuous in $u$ and $v_x$, see \cite{Hoff86}, provided the Rankine-Hugoniot conditions 
$$
\begin{aligned}
-s [u] &= [v]
\\
- s [v] &= \left [ \sigma (u) +\frac{v_x}{u} \right ]
\end{aligned}
$$
are satisfied, where $s$ is the shock speed and $[q] = q_+ - q_-$ the jump of the quantity $q$. As $v$ is continuous, 
the shocks are stationary $s=0$ and $[u] \ne 0$ has to satisfy
\begin{equation}\label{sec7RH}
s= 0 \, , \quad \left [ \sigma (u) + \frac{u_t }{u} \right ] = 0 \, .
\end{equation}
\end{itemize}

We construct a family of periodic solutions to \eqref{longonev} defined for $(t,x) \in [1,2]\times \R$. Fix states $a, b$ satisfying 
$0 < a < 2a < b < 2b$ and suppose that \eqref{condnonm} is satisfied. We denote by $S(t)$ the common value
\begin{equation*}
S(t) := \frac{1}{t a} a  + \sigma (t a) = \frac{1}{t b} b  + \sigma (t b)  \, ,  \qquad 1 \le t \le 2 \, .
\end{equation*}
Next, fix $0 < \theta < 1$ and define the periodic function
\begin{equation}\label{perU}
U(t,x)  := \begin{cases}
 a \,  t   &  \quad  \; \;  k <  x <  k + \theta
 \\
 b \,  t   &  k + \theta < x < k + 1
 \end{cases}
\quad \; k \in \mathbb{Z}  
\end{equation}
Finally, let $c_\theta = \theta a + (1-\theta) b$ and set
\begin{equation}
\label{perY}
\begin{aligned}
Y(t,x) = \int_0^x U (t, y) dy &= 
\begin{cases}
k c_\theta t + (x-k) a t  & \quad  \; \;  k <  x <  k + \theta
\\
 k c_\theta t + \theta a t + \big ( x - k - \theta) b t   &  k + \theta < x < k + 1
\end{cases} \, , \quad k \in \mathbb{Z} \, ,
\\
V(t,x) = \del_t Y (t,x) &=
\begin{cases}
k c_\theta  + (x-k) a   & \quad  \; \;  k <  x <  k + \theta
\\
 k c_\theta  + \theta a  + \big ( x - k - \theta) b    &  k + \theta < x < k + 1
\end{cases} \, , \quad k \in \mathbb{Z} \, .
\end{aligned}
\end{equation}
Then $(U, V)$ is a weak solution of \eqref{longonev} on $ [1,2]\times\R$.
It satisfies the equations in a classical sense on $(1,2)\times (k, k + \theta)$, $(1,2)\times(k + \theta , k+1)$ and 
the Rankine-Hugoniot conditions \eqref{sec7RH} at the interfaces $x = k$ and  $x = k + \theta$ for $1 \le t \le 2$.

The function $Y(t,x)$ is then rescaled and restricted to the interval $(t,x) \in  Q = (1,2)\times (-1,1)$ to define
\begin{equation}
\label{exsoln3}
y_n (t,x) = \frac{1}{n} Y(t, nx)  \, \quad u_n = \del_x y_n = U  (t, nx)  \, \quad {v_n} = \del_t y_n  =  \frac{1}{n}  V (t, nx)
\end{equation}
One checks that $(u_n , v_n )$ is a weak solution of \eqref{longonev} which is stationary for the momentum equation.
Moreover, one easily computes the limits
$$
\begin{aligned}
u_n  &\rightharpoonup  (a\theta + b(1-\theta) ) t  \quad \mbox{ weakly-$\star$ in  $L^\infty \big ( Q \big )$ }
\\
\del_x v_n &\rightharpoonup  (a\theta + b(1-\theta) )  \quad \mbox{ weakly-$\star$ in  $L^\infty \big ( Q \big )$ }
\\
v_n &\to (a\theta + b(1-\theta)) x  \quad \mbox{ strongly in $L^2 \big ( Q \big )$ }
\end{aligned}
$$
and 
$$
\sigma (u_n) \rightharpoonup  \theta \sigma ( at) + (1-\theta) \sigma(bt) \ne \sigma \big (  \theta a t + (1-\theta) b t) \, ,
$$
weak-$\star$ in $L^\infty(Q)$. The oscillations in solutions are characterized by the Young measure 
$$\nu = \theta \delta_{at} + (1-\theta) \delta_{b t}$$
and are induced by oscillations in the initial data $u_n (1,x)$ and $\del_x v_n (1, x)$.

\subsection{A class of nonlinear models describing  viscoplastic response}\label{sec:mdvisco}
Consider next the equation
$$
y_{t t} = \big ( \sigma (y_x , y_{x t} ) \big )_x 
$$
which may be expressed as a system in the form
\begin{equation}
\label{viscopl}
\begin{aligned}
u_t &= v_x
\\
v_t &=  \del_x \big ( \sigma(u, v_x ) \big )
\end{aligned}
\end{equation}
Such systems have been used as models of viscoplasticity describing high strain-rate deformations
and are hyperbolic parabolic under the assumption $\sigma_q (p, q) > 0$.

This system has analogous properties to the example of section \ref{sec:1dnonmon}, namely:
(i)  It admits uniform shearing solutions \eqref{unifshear} for any $\kappa \in \R$
independently of the form of the function $\sigma (p, q)$.
(ii) It can have  piecewise smooth solutions which satisfy the system in regions separated by stationary  interfaces,
and across the interfaces they satisfy  $[v] = 0$, $[u] \ne 0$ and $[v_x] \ne 0$ and the Rankine-Hugoniot jump conditions
\begin{equation}\label{sec7bRH}
s= 0 \, , \quad \left [ \sigma (u, v_x)  \right ] = 0 \, .
\end{equation}

To continue, let $a, b$ be two fixed states satisfying $0 < a < 2 a < b < 2 b$ and suppose that 
\begin{equation}\label{hypvp}
\sigma ( \tau a , a) = \sigma (\tau b , b ) \quad \mbox{for $\tau \in [1, 2]$} \, .
\end{equation}
We may then define periodic solutions  $Y(t, x)$ and $(U,V)$  for  \eqref{viscopl} defined by \eqref{perU}, \eqref{perY} on the domain
$(t, x) \in (1,2) \times \R$. Then $\{ y_n \}$ defined via the rescaling \eqref{exsoln3} produces an oscillatory solution of \eqref{viscopl}
with sustained oscillations induced by the initial data.

To see what the hypothesis \eqref{hypvp} signifies, consider the special case $\sigma (p, q) = \varphi (p) q^n$ where $n > 0$ exponent and $\varphi (u) > 0$.
This satisfies $\sigma_q (p,q) > 0$ so the system \eqref{viscopl} falls within the class of hyperbolic-parabolic systems. For this class \eqref{hypvp} dictates
$$
\varphi (\tau a) = \varphi (\tau b) \Big ( \frac{b}{a}\Big )^n   \quad \tau \in [1,2]
$$
and one can select $\varphi (u)$ that satisfies that for $a,b$ as above.


\section{ Quasilinear systems in gas dynamics for a viscous adiabatic gas}\label{sec:1dthermoel}
The system of compressible gas dynamics in one space dimension has the form
\begin{equation}
\label{gasdyn}
\begin{aligned}
u_t - v_x &= 0
\\
v_t - \sigma(u,\theta)_x &= 0
\\
\big ( \tfrac{1}{2} v^2 + e(u, \theta) \big )_t - ( \sigma(u,\theta) \, v )_x 
&= 0
\end{aligned}
\end{equation}
The constitutive theory is determined by a free energy function $\psi = \psi (u, \theta)$ via the formulas
\begin{equation}
\label{const}
\sigma = \frac{ \del \psi}{\del u} \, , \quad \eta = - \frac{ \del \psi}{\del \theta}  \, , \quad e = \psi + \theta \, \eta
\end{equation}
and satisfies the Maxwell relations
\begin{equation}
\label{maxwell}
\sigma_\theta =  - \eta_u  \, , \quad  e_\theta = \theta \eta_\theta \, , \quad e_u = \sigma - \theta \sigma_\theta \, .
\end{equation}
The relations \eqref{const} are motivated from thermodynamics and (are designed to) induce an additional conservation equation
expressing conservation of the entropy $\eta(u,\theta)$ 
for smooth solutions of  \eqref{gasdyn},
$$
\del_t \eta(u,\theta) = 0 \, .
$$
The system \eqref{gasdyn}  may be written as a system of conservation laws 
$$
\del_t A(U) + \del_x F(U) = 0
$$
where
\begin{equation*}
\begin{split}
U  =  
\begin{pmatrix}
u \\ v \\ \theta
\end{pmatrix}
\quad
A(U) =
\begin{pmatrix}
u \\ v \\ \tfrac{1}{2} v^2 + e(u,\theta)
\end{pmatrix}
\quad
F(U) =
- \begin{pmatrix}
v \\ \sigma (u, \theta)  \\ v \, \sigma(u, \theta)
\end{pmatrix}
\end{split}
\end{equation*}
It is hyperbolic if
\begin{equation}\label{hyperb}
e_\theta > 0 \, ,  \quad \sigma_u > 0 \, ,  \quad \eta_\theta > 0 \, .
\tag{H}
\end{equation}
with the wave speeds computed by $\lambda_0 = 0$ and $\lambda_\pm = \pm \sqrt{ \sigma_u + \tfrac{\sigma_\theta^2}{\eta_\theta}}$.
The condition of hyperbolicity \eqref{hyperb} is equivalent to the hypothesis $\psi_{uu} > 0$ and $\psi_{\theta \theta} < 0$.

We will consider the one-dimensional hyperbolic-parabolic system 
\begin{equation}
\label{vhcg}
\begin{aligned}
u_t - v_x &= 0
\\
v_t - \sigma(u,\theta)_x &= ( \frac{\mu}{u} v_x )_x 
\\
\big ( \tfrac{1}{2} v^2 + e(u, \theta) \big )_t - \Big ( \sigma(u,\theta) \, v \Big )_x 
&= \Big ( \frac{\mu}{u} \,  v \, v_x \Big )_x +  \Big ( \frac{\kappa}{u} \theta_x \Big )_x  \, .
\end{aligned}
\end{equation}
In this model $u$ stands for the specific volume (the inverse of the density), $v$ is the longitudinal velocity,
and $\theta$ the temperature,  while the internal energy $e$ and  the stress $\sigma$ are determined via constitutive relations \eqref{const}, \eqref{maxwell};
in this interpretation $u > 0$ and $\theta > 0$. 
In \eqref{vhcg} we adopted a Stokes law for the viscous stress
$$
\tau_{viscous} = \mu (u,\theta) \frac{1}{u} v_x \quad \mbox{ with $\mu (u,\theta) \ge 0$},
$$
and a Fourier law for the heat conduction,
$$
Q = \kappa(u, \theta) \frac{1}{u}   \theta_x \quad \mbox{ with $\kappa (u,\theta) \ge 0$}.
$$
The form of $\frac{\mu}{u}$ for the viscosity and $\frac{\kappa}{u}$ for the heat conductivity 
is due to the interpretation of the model in Lagrangian coordinates.
The theory of gas dynamics for viscous and heat-conducting gases satisfies the entropy production identity
\begin{equation}
\label{entropyp}
\del_t \eta (u,\theta) - \del_x \left ( \frac{Q}{\theta} \right ) = \frac{\mu}{u} \frac{v_x^2}{\theta} + \frac{\kappa}{u}  \left ( \frac{\theta_x}{\theta} \right )^2  \ge 0 \, .
\end{equation}

Another interpretation of  \eqref{vhcg}
is to describe  one dimensional shear motions of a thermoviscoelastic materials.  In this case $u$ will be the shear strain, 
$v$ the velocity in the shear direction, $\theta > 0$ the temperature; 
in this interpretation $u$ does not obey a positivity constraint, while the viscous stress will be $\tau = \mu v_x$, the heat conductivity 
$q = \kappa \theta_x$, and \eqref{entropyp} has to be adapted accordingly.

Our objective is to construct a periodic solution of \eqref{vhcg}, defined for $1 \le t \le 2$ and $x \in \R$ periodic with period $p =1$, in analogy to the construction of section \ref{sec:1dnonmon}.  This by rescaling will produce a weakly convergent sequence of exact solutions to \eqref{vhcg} that will converge weakly but not strongly.
The construction will be done for $\mu$ constant and for $\kappa = 0$, that is for a {\it viscous adiabatic gas}, and will be based on combining uniform extension solutions with 
functions that exhibit jumps on the temperature, strain, and strain-rate but with continuous velocities.

\subsection{Uniform extension solutions}
Consider \eqref{vhcg} with $\mu$ and $\kappa$ both constant with $\mu > 0$ and $\kappa \ge 0$. We study a special class of exact solutions
$(\bar y (x,t) , \bar \theta (t) )$ which are required to be of the form
\begin{equation}\label{us-tve}
\bar y (x,t ) = a x t \, , \qquad  \bar \theta = \bar \theta (t )
\end{equation}
where $a > 0$ is a constant and $\bar \theta(t)$ depends on time only. In this case
$$
\bar v (x, t) = a x \, , \quad \bar u (x,t) = \bar u (t) = a t 
$$
Such a solution describes uniform extension of the gas when $at > 1$ or compression when $a t < 1$. One easily checks that \eqref{us-tve}
solves \eqref{vhcg} if  $\bar \theta (t) = \bar \theta (t ; a, \varphi)$ solves tha initial value problem
\begin{equation}\label{ivp-us}
\left \{
\begin{aligned}
\bar \theta \del_t \eta ( \bar u , \bar \theta ) &= \frac{\mu}{\bar u}  (\bar u_t )^2
\\
\; \; \bar \theta (t =1  ; a, \varphi) &= \varphi
\end{aligned}
\right .
\end{equation}
where $\varphi > 0$ depicts the temperature at time $t = 1$ which is assumed constant. The latter equation follows from \eqref{entropyp} which
is equivalent to \eqref{vhcg}$_3$ for smooth processes.

In the sequel we employ a constitutive theory for the gas 
\begin{equation}\label{assform}
\psi (u, \theta) = \theta \int_1^u \tau (s) \, ds - \big ( \theta \ln \theta - \theta \big )
\tag {CF}
\end{equation}
whence
$$
\sigma = \theta \tau (u)  \, , \quad \eta (u, \theta) = \ln \theta - W (u)
$$
with $W(u) = \int_1^u \tau (s) \, ds $.  Then, the initial value problem \eqref{ivp-us} reduces to asking that 
$\bar \theta (t ; a, \varphi)$ solves the problem
\begin{equation}\label{ivp-uss}
\left \{
\begin{aligned}
\frac{d \bar \theta}{dt} - a  \, \tau (at) \bar \theta  &= \frac{\mu}{\bar u}  (\bar u_t )^2 = \frac{\mu a}{t}
\\
\; \; \bar \theta (t =1 ) &= \varphi
\end{aligned}
\right .
\end{equation}
The solution in the interval of interest $t \in [1,2]$ has the explicit form
$$
\bar \theta (t ; a, \varphi) = \varphi e^{W(at) - W(a)} + e^{W(at)} \int_1^t \frac{\mu a}{s} e^{- W(as)} \, ds
$$
Observe that since $\mu, a, \varphi > 0$ the function $\bar \theta (t) > 0$ as well.

\subsection{Discontinuous Solutions}
Next consider the problem \eqref{vhcg} now under adiabatic conditions $\kappa = 0$ and we look for solutions which across
an interface are continuous in velocity, $[ v ] = 0$ but have jump discontinuities $[ \theta] \ne 0$, $[u ] \ne 0$ and $[ v_x] \ne 0$.
The Rankine-Hugoniot conditions,
$$
\begin{aligned}
- s [u] &= [v] 
\\
- s [v] &= [ \sigma (u,\theta) + \frac{\mu}{u} v_x ]
\\
-s [ \tfrac{1}{2} v^2 + e(u,\theta) ] &=  \big [  (\sigma (u,\theta) + \frac{\mu}{u} v_x) v ]
\end{aligned}
$$
reduce to the requirement that
\begin{equation}\label{interfcond}
s = 0 \, , \quad [v] = 0 \, , \quad \Big [ \sigma (u,\theta) + \frac{\mu}{u} v_x \Big ] = 0
\end{equation}
If two solutions are smooth on the domains right and left of an interface located at $x = x_0$ and satisfy across the interface the
jump conditions \eqref{interfcond} then they give rise to a weak solution of the problem \eqref{vhcg}.

\subsection{Construction of an oscillatory solution}
We construct a weak solution consisting of two uniform extension solutions connected  via a jump discontinuity satisfying \eqref{interfcond}
along an interface at $x = 0$.

\begin{lemma}
Let the constitutive theory be of the form \eqref{assform}, and let $A, B$ satisfy $0 < A < 2A < B $ and $\varphi_A , \varphi_B > 0$ be given
and suppose that $\tau (u) > 0$ for $u \in [A, 2A]$.
Define for $t \in [1,2]$
$$
u_A  = A t \, , \quad \theta_A = \bar \theta (t ; A, \varphi_A) \, , \qquad u_B = B t \, ,  \quad \theta_B = \bar \theta (t ; B , \varphi_B)
$$
where $\bar \theta (t, a , \varphi)$ solves the initial value problem \eqref{ivp-uss}. Then if  $\tau (\cdot)$ is selected to satisfy
\begin{equation}\label{interfcond2}
\tau (At) \theta_A (t) = \tau (B t)  \theta_B ( t) \quad t \in [1,2]
\end{equation}
then  \eqref{interfcond} is satisfied and 
\begin{equation}\label{spsoln1}
(u, v, \theta) = \begin{cases} (At, A x , \theta_A (t))   & x < 0 \\  ( Bt, B x , \theta_B (t))  & x > 0 \end{cases}
\end{equation}
is a weak solution of \eqref{vhcg} with $\kappa = 0$ for $t \in [1,2]$.
\end{lemma}

\begin{proof}
Let $\tau(\cdot)$ de defined in $[A, 2A]$ and let $\theta_A(t)$ be defined by \eqref{ivp-uss} with $a = A$ and $\varphi = \varphi_A$.
Let 
$$
g(t) = \tau (At) \theta_A (t)
$$
and define $\hat \theta (t)$ to be the solution of the initial value problem
$$
\begin{aligned}
\frac{d \hat \theta}{dt} &= \frac{\mu B}{t} + B \tau(At) \theta_A (t)
\\
\hat \theta ( t =1 ) &= \varphi_B
\end{aligned}
$$
Note that $\hat \theta (t) > 0$.

Given $\hat \theta (t)$ we select $\tau (\cdot)$ in the interval $[B, 2B]$ so that 
$$
\tau(Bt) = \tau(At) \frac{\theta_A (t)}{\hat \theta (t)}
$$
Since $2A < B$ the  function $\tau(\cdot)$ is well defined on $[B, 2B]$. Moreover, $\hat \theta$ solves \eqref{ivp-uss} and thus
$\hat \theta = \theta_B$. We conclude by noticing that \eqref{interfcond2} implies the interface condition
\eqref{interfcond} across the interface $x = 0$ and thus \eqref{spsoln1} is a weak solution.
\end{proof}

The above argument is symmetric in the placement of $A$ and $B$ and thus
\begin{equation}\label{spsoln2}
(u, v, \theta) = \begin{cases} (Bt, B x , \theta_B (t))   & x < 0 \\  ( At, A x , \theta_A (t))  & x > 0 \end{cases}
\end{equation}
is again a weak solution of \eqref{vhcg} with $\kappa = 0$ for $t \in [1,2]$. We conclude by iterating these two
solutions to construct an oscillating weak solution defined on $\R \times [1,2]$.
We proceed along the lines of section \ref{sec:1dnonmon}.

Fix states $A, B$, $\theta_A (t)$, $\theta_B (t)$ as above and suppose that \eqref{interfcond2} is satisfied. 
Let $0 < \lambda < 1$, let $c_\lambda = \lambda A + (1-\lambda) B$ and define
\begin{equation}\label{perU2}
\begin{aligned}
u(t,x)  &:= \begin{cases}
 A \,  t   &  \quad  \; \;  k <  x <  k + \lambda
 \\
 B \,  t   &  k + \lambda < x < k + 1
 \end{cases}
 \quad \; k \in \mathbb{Z}  
 \\
 \theta(t,x) &:= \begin{cases}
 \theta_A (t)    &  \quad  \; \;  k <  x <  k + \lambda
 \\
 \theta_B (t)    &  k + \lambda < x < k + 1
 \end{cases}
\quad \; k \in \mathbb{Z}  
\end{aligned}
\end{equation}
Next,  set
\begin{equation}
\label{perY2}
\begin{aligned}
y(t,x) = \int_0^x u (t, z) dz &= 
\begin{cases}
k c_\lambda t + (x-k) A t  & \quad  \; \;  k <  x <  k + \lambda
\\
 k c_\lambda t + \lambda A  t + \big ( x - k - \lambda) B t   &  k + \lambda < x < k + 1
\end{cases} \, , \quad k \in \mathbb{Z} \, ,
\\
y(t,x) = \del_t y (t,x) &=
\begin{cases}
k c_\lambda + (x-k) A    & \quad  \; \;  k <  x <  k + \lambda
\\
 k c_\lambda  + \lambda A + \big ( x - k - \lambda) B    &  k + \lambda < x < k + 1
\end{cases} \, , \quad k \in \mathbb{Z} \, .
\end{aligned}
\end{equation}
Then $(u, v , \theta)$ is a weak solution of \eqref{longonev} on $ [1,2]\times\R$ satisfying 
the equations a classical sense on $(1,2)\times (k, k + \lambda)$, $(1,2)\times(k + \lambda , k+1)$ and 
the Rankine-Hugoniot conditions at the interfaces $x = k$,  $x = k + \lambda$ for $k \in\mathbb{Z}$  and  $1 \le t \le 2$.

Finally, by rescaling 
\begin{equation}
\label{exsoln4}
y_n (t,x) = \frac{1}{n} y(t, nx)  \, \quad u_n = \del_x y_n = u (t, nx)  \, \quad {v_n} = \del_t y_n  =  \frac{1}{n}  v (t, nx) \,  \quad 
\theta_n (t, x) = \theta (t, nx)
\end{equation}
and restricting to the interval $(t,x) \in  Q = (1,2)\times (-1,1)$ we obtain a solution with sustained oscillations 
propagating from the data at $t=1$ to the solution in the interval $t \in [1,2]$.
One checks the limits
$$
\begin{aligned}
u_n  &\rightharpoonup  (A \lambda + B (1-\lambda) ) t  \quad \mbox{ weakly-$\star$ in  $L^\infty \big ( Q \big )$ }
\\
\del_x v_n &\rightharpoonup (A \lambda + B (1-\lambda) )  \quad \mbox{ weakly-$\star$ in  $L^\infty \big ( Q \big )$ }
\\
v_n &\to ((A \lambda + B (1-\lambda) ) x  \quad \mbox{ strongly in $L^2 \big ( Q \big )$ }
\\
\theta_n &\rightharpoonup  \big(\theta_A (t) \lambda + \theta_B (t) (1 -\lambda)\big)  \quad \mbox{ weakly-$\star$ in  $L^\infty \big ( Q \big )$ }
\end{aligned}
$$


\section{The compressible Navier-Stokes system in one space dimension}\label{sec:comprns}

In the final section we consider the one-dimensional version of \eqref{eq:compNS} in Eulerian coordinates:
\begin{equation}
\label{compNS}
\begin{aligned}
\rho_t  + (\rho u)_y  &= 0
\\
(\rho u)_t  + ( \rho u^2  +  p (\rho) )_y  &=  \mu u_{y y}
\end{aligned}
\end{equation}
where $\rho (t,y)$ and $u(t,y)$ are the density and velocity of the fluid expressed in Eulerian coordinates $(t,y)$ while $\mu > 0 $ is the viscosity.

The Rankine-Hugoniot jump conditions across a jump discontinuity that moves with speed $s$
have the form
\begin{equation}\label{RHgd}
\begin{aligned}
\big [ \rho (u-s) \big ] &= 0
\\
\big [ \rho u (u-s) \big ] &= \big [ - p(\rho) + \mu u_y \big ]
\end{aligned}
\end{equation}
It follows that there can be solutions that are continuous in velocity $[ u ] = 0$ but discontinuous in density, $[\rho] \ne 0$, 
across the interface moving with speed $u = s$. Such solutions will satisfy the conservation of mass on the interface \eqref{RHgd}$_1$.
The conservation of momentum reduces to the condition
\begin{equation}\label{RHmom}
\big [ - p(\rho) + \mu u_y \big ] = 0
\end{equation}
Observe that since $[ u ] = 0$ across the interface the jump $[u_y]$ is well defined.

\subsection{Langrangian to Eulerian}
The system \eqref{longonev}, which we now write in the notation
\begin{equation}
\label{longonev1}
\begin{aligned}
w_t &= v_x
\\
v_t &= \sigma(w)_x +  \del_x \left ( \frac{\mu}{w} v_{x} \right )
\end{aligned}
\end{equation}
is the Lagrangian form of the equation \eqref{compNS}.  The equivalence is effected via the transformation
$$
w = \frac{\del y}{\del x} \, , \quad v = \frac{\del y}{\del t}
$$
and setting
\begin{equation}\label{transf}
\frac{\del y}{\del t} (t,x) = v (t, x) = u (t, y(t,x)) \, , \quad \rho (t, y(t,x))  = \frac{1}{\frac{\del y}{\del x} (t,x)} = \frac{1}{w(t,x)}
\end{equation}
A direct calculation shows that for smooth solutions \eqref{longonev1} transforms to \eqref{compNS}. 
The same result can be achieved by a transformation of the weak forms provided that $y(t,x)$ is assumed to be a bi-Lipschitz homeomorphism, 
\cite{b-dafermos16}. 

Consider the special periodic solution of \eqref{longonev} that has the form \eqref{perU}, \eqref{perY}. Recall, that $a< 2a< b$ and that
for this solution
\begin{equation}\label{eqny}
\begin{aligned}
y(t,x) &= 
\begin{cases}
k v_0(\theta)  t + (x-k) a t  & \quad  \; \;  k <  x <  k + \theta
\\
 k v_0(\theta)  t + \theta a t + \big ( x - k - \theta) b t   &  k + \theta < x < k + 1
\end{cases} \, , \quad k \in \mathbb{Z} \, ,
\\
v(t,x)  &=
\begin{cases}
k v_0(\theta)  + (x-k) a   & \quad  \; \;  k <  x <  k + \theta
\\
 k v_0(\theta)  + \theta a  + \big ( x - k - \theta) b    &  k + \theta < x < k + 1
\end{cases} \, , \quad k \in \mathbb{Z} \, .
\end{aligned}
\end{equation}
where $v_0 (\theta) = \theta a + (1-\theta) b$ and $0 < \theta < 1$. Also, recall that
\begin{equation}\label{eqnw}
\begin{aligned}
w(t,x) &= 
\begin{cases}
t a & \quad  \; \;  k <  x <  k + \theta
\\
 t b  &  k + \theta < x < k + 1
\end{cases} \, , \quad k \in \mathbb{Z} \, ,
\end{aligned}
\end{equation}

When we transform \eqref{eqny} and \eqref{eqnw} to Eulerian coordinates via the transformation \eqref{transf}
we obtain the following Eulerian form
\begin{align}
u(t,y) &= \frac{y}{t}
\label{eqnu}
\\
\rho(t,y) &= 
\begin{cases} \;  \frac{1}{ta}  \quad & \quad  0 < \frac{y}{t}  - k v_0 (\theta) <  a \theta   \\[5pt]
                      \;  \frac{1}{tb}  \quad &  \; \; a \theta < \frac{y}{t}  - k v_0 (\theta) < v_0 (\theta)
\end{cases}
\quad k \in \mathbb{Z}
\label{eqnrho}
\end{align}
\begin{remark} The function
$$
\bar \rho (t, y) = \frac{1}{ta} \, , \quad \bar u (t,y) = \frac{y}{t}
$$
with $a > 0$ is an exact solution of the system \eqref{compNS}
\end{remark}

For the solution \eqref{eqnu}-\eqref{eqnrho} we note that (i) it has discontinuities in $\rho$ on the lines $\xi_k (t) = k t v_0 (\theta)$ and 
$\xi_{k, \theta} = k t v_0 (\theta) + t a \theta$. (ii) On the line $\xi_k (t)$ we have
$$
\frac{d \xi_k}{dt} = k v_0(\theta) = u (t,y) \Big |_{y = \xi_k (t)}
$$
(iii) On the line $\xi_{k, \theta} (t)$ we have
$$
\frac{d \xi_k}{dt} = k v_0(\theta)  + a \theta = u (t,y) \Big |_{y = \xi_{k, \theta} (t)}
$$
(iv) Let us denote 
$$I_k = \{ (t, y) : 0 < y - k t v_0 (\theta) < t a \theta \} \, ,  \quad 
J_k = \{ (t, y) : t a \theta < y - k t v_0 (\theta) < t v_0 (\theta) \} \, .
$$ 
Then we compute on each domain $I_k$ and $J_k$ that
$$
\del_t \rho + \del_y (\rho u) = 0
$$
and 
$$
\begin{aligned}
\del_t (\rho u)  + \del_y (\rho u^2 ) &= 0 
\\
\del_y \Big ( - p(\rho) + \mu \frac{\del u}{\del y} \Big ) = 0
\end{aligned}
$$
that is the equations \eqref{compNS} are satisfied on each of $I_k$, $J_k$. 
(v) Suppose that we impose the hypothesis on the pressure
\begin{equation}\label{hyppr}
p \left ( \frac{1}{at} \right ) = p \left ( \frac{1}{bt} \right )  \quad \mbox{for \;  $t \in [1,2]$}
\tag {AP}
\end{equation}
then on the interface $\xi_k (t) = k t v_0 (\theta)$ we have
$$
[ \rho u (u - \dot \xi_k) ] = 0 ,   [ - p (\rho) + \mu u_y  ] = 0
$$
The same relations also hold on the interface $\xi_{k,\theta} (t) = k t v_0 (\theta) + t a \theta$. We thus conclude

\begin{proposition}
Let the pressure satisfy \eqref{hyppr}. Then if $0< a < 2 a < b < 2b$ then \eqref{eqnu}-\eqref{eqnrho} define a weak solution 
of the compressible Navier-Stokes system \eqref{compNS}
defined for $y \in \R$,  $t \in [1,2]$.
\end{proposition}

For states $a$, $b$ satisfying $0< a < 2a < b < 2b$ the condition \eqref{hyppr} can be fulfilled and results 
to nonmonotone equations for the pressure.

\subsection{ Sustained oscillations}
Next define the sequence of solutions $(\rho_n , u_n )$ by
\begin{equation}\label{oscisol}
u_n (t, y) = \frac{1}{n} u (t, n y) \, , \quad \rho_n (t, y) = \rho (t, n y)   \qquad y \in (-1,1), \;  t \in [1,2] \, ,
\end{equation}
where $(\rho, u)$ is given by \eqref{eqnu}-\eqref{eqnrho}, $n \in \mathbb{N}$. Then we have
$$
u_n (t,y) = \frac{y}{t} 
$$
and
$$
\rho_n (t,y) = 
\begin{cases} \;  \frac{1}{ta}  \quad &  \qquad  \quad  \frac{k}{n} v_0 (\theta) < \frac{y}{t}  <  \frac{k}{n}  v_0 (\theta) + \frac{1}{n}  a \theta   \\[5pt]
                      \;  \frac{1}{tb}  \quad & \frac{k}{n} v_0 (\theta) + \frac{1}{n}  a \theta   < \frac{y}{t}   <\frac{k+ 1}{n}  v_0 (\theta) 
\end{cases}
\quad k \in \mathbb{Z}
$$

It is easy to check that $(\rho_n, u_n)$ is a weak solution of \eqref{compNS}. Indeed, since as we saw $(\rho, u)$ is a weak solution of \eqref{compNS}
on $[1,2] \times \R$, we have for $\phi \in C^\infty_c ([1,2)\times \R)$
\begin{equation}\label{weak1}
\begin{aligned}
\iint \rho(t,y) \del_t \phi (t,y) dt dy + \iint \rho(t,y) u (t,y)  \del_y \phi (t,y) dt dy + \int \rho (1,y) \phi(1,y) dy = 0 \, .
\end{aligned}
\end{equation}
We introduce the test function $\phi (t, y) = \frac{1}{n} \varphi (t, \frac{1}{n} y )$ in \eqref{weak1}, perform the change of variables $y' = \frac{1}{n} y$
and use \eqref{oscisol} to obtain
\begin{equation}\label{newweak1}
\begin{aligned}
\iint \rho_n (t,y) \del_t \varphi (t,y) dt dy + \iint \rho_n (t,y) u_n (t,y)  \del_y \varphi (t,y) dt dy + \int \rho_n  (1,y) \varphi(1,y) dy = 0 \, ,
\end{aligned}
\end{equation}
that is $(\rho_n, u_n)$ satisfies the weak form of \eqref{compNS}$_1$. A similar calculation shows that $(\rho_n, u_n)$ also satisfies the weak form
of \eqref{compNS}$_2$.

If the sequence $(\rho_n, u_n)$ is restricted on $Q = [1,2]\times (-1,1)$ we see that
$$
\rho_n \rightharpoonup \rho = \theta \frac{1}{ta} + (1-\theta) \frac{1}{tb} \quad \mbox{ weakly-$\star$ in  $L^\infty \big ( Q \big )$ }
$$
Of course $u_n = \frac{y}{t}$ converges strongly.

\begin{theorem}
Under the hypothesis \eqref{hyppr} there is a sequence of weakly converging initial data so that the solution $(\rho_n, u_n)$ of
\eqref{compNS} has the property that $\rho_n \rightharpoonup \rho$ weakly while $u_n \to u$ strongly.
\end{theorem}

\begin{remark}
Note that $(\rho_n, u_n)$ are also weak solutions of the pressureless Euler equations, and 
the assumption \eqref{hyppr} guarantees that the total stress is a function of time.
\end{remark}

\bigskip
{\bf Statements and Declarations}

{\bf Data availability} 
The data supporting the findings of this study are available within the paper,

{\bf Competing Interests}
The author has no relevant competing interests to declare.

{\bf Funding} 
Research supported by KAUST baseline funds, No BAS/1/1652-01-01

%
%
%
%
%
%




%
\end{document}